\documentclass[12pt,a4paper,reqno]{amsart}
\usepackage{mystyle}

\def\firstauth#1{}
\def\second#1{}
\def\third#1{}
\begin{document}

\title{Multi-cyclic graphs in the random graph process with restricted budget}

\author[I\v{l}kovi\v{c}]{Daniel I\v{l}kovi\v{c}}
\address{Faculty of Informatics, Masaryk University, Botanick\'a 68A, 602 00 Brno, Czech Republic.}
\email{493343@mail.muni.cz}
\thanks{(DI, XS): Research supported by the MUNI Award in Science and Humanities (MUNI/I/1677/2018) of the Grant Agency of Masaryk University.}

\author[León]{Jared León}
\email{jared.leon@warwick.ac.uk}
\address{Mathematics Institute and DIMAP, University of Warwick, Coventry, UK}
\thanks{(JL): Research supported by the Warwick Mathematics Institute Centre for Doctoral Training.}
  
\author[Shu]{Xichao Shu}
\email{540987@mail.muni.cz}

\begin{abstract}

  We study a controlled random graph process introduced by Frieze, Krivelevich, and Michaeli. In this model, the edges of a complete graph are randomly ordered and
  revealed sequentially to a builder. For each edge revealed, the builder must
  irrevocably decide whether to purchase it. The process is subject to two
  constraints: the number of observed edges~$t$ and the builder's budget~$b$. The
  goal of the builder is to construct, with high probability, a graph possessing a
  desired property.

  Previously, the optimal dependencies of the budget $b$ on $n$ and $t$ were established for constructing a graph containing a
  fixed tree or cycle, and the authors claimed that their proof could be extended to
  any unicyclic graph. The problem, however, remained open for graphs containing at
  least two cycles, the smallest of which is the graph~$K_4^-$ (a clique of size four
  with one edge removed).

  In this paper, we provide a strategy to construct a copy of the graph~$K_4^-$
  if~$b \gg \max\left\{n^6 / t^4, n^{4 / 3} / t^{2 / 3}\right\}$, and show that this
  bound is tight, answering the question posed by Frieze et al. concerning this specific
  graph. We also give a strategy to construct a copy of a graph consisting of~$k$
  triangles intersecting at a single vertex (the $k$-fan)
  if~$b \gg \max\left\{n^{4k - 1} / t^{3k - 1}, n / \sqrt{t}\right\}$, and also show that
  this bound is tight. These are the first optimal strategies for
  constructing a multi-cyclic graph in this random graph model.
  
\end{abstract}

\maketitle

\section{Introduction}
Random graph processes have garnered significant attention in recent years, as they offer a natural framework for modeling the evolution of complex networks over time. These processes do not only provide a straightforward, and in some cases the only, method for constructing counterexamples in extremal graph theory but also exhibit strong connections to algorithmic graph theory. In particular, the analysis of random graph processes can lead to valuable insights into the typical running times of algorithms.

Random graph processes can be defined in various ways, with one of the most well-known and the simplest being the Erd\H{o}s–R\'enyi random graph process, introduced by Erd\H{o}s and R\'enyi~\cite{ER1959, ER1960}. This process begins with an empty graph on $n$ vertices. At each step, an edge is chosen uniformly at random from the set of non-edges and is added to the graph. 

At any fixed time $t$, the random graph process is distributed as a uniform random graph $G(n,t)$. A \defin{graph property} is a family of graphs that is invariant under isomorphisms. Such a property is called \defin{monotone} if it remains preserved under the addition of edges. The \defin{hitting time} for a monotone property $P$ is defined as the random variable $t_P=\min\set{t\in [N]:\text{$G(n,t)$ satisfies $P$ }}$. A significant body of research focuses on identifying the hitting time for various monotone graph properties in the random graph process. 

\subsection{Random graph process with restricted budget}

Beyond the classical binomial random graph model, also known as the Erd\H{o}s–R\'enyi
random graph, numerous fascinating variants have been investigated. Examples include the~$H$-free process \cite{TP2010,ESW1995,DA2021, Lutz2011}, with particular attention
given to the triangle-free process~\cite{Bohman2009, TP2021, FGGM2020}, the more recently studied model of
multi-source invasion percolation on the complete graph~\cite{ALB2023}, Achlioptas processes proposed by Dimitris Achlioptas \cite{SAW2014, TA2001, OL2012, OL2017, JN2007}, 
and the semi-random graph
process suggested by Peleg Michaeli~\cite{BHKPSS2020}, and studied recently in~\cite{BMPR2024, BGHK2020, GKP2024, GKMP2022, GMP2022, GH2023, PS2024}.

Recently, Frieze, Krivelevich, and Michaeli \cite{frieze2023} introduced the following controlled random graph process. Consider a Builder who observes the edges of a random permutation $e_1, \ldots, e_N$ one at a time and must make an irrevocable decision whether to select $e_i$ upon seeing it. A \((t, b)\)-strategy is an online algorithm (either deterministic or randomized) that Builder employs in this setting, where observes only the first $t$ edges of the random permutation of $E(K_n)$ and is permitted to select at most $b$ of them. The Builder’s goal is for her graph to achieve a specified monotone graph property of $P$. For instance, there exists a simple \((t_{\text{con}}, n-1)\)-strategy ensuring that Builder's graph becomes connected at the hitting time for connectivity: Builder selects an edge if and only if it reduces the number of connected components in the graph.

In this model, the problems of constructing spanning structures such as a perfect matching, a Hamilton cycle, a graph with a prescribed minimum degree, or a graph with a fixed level of connectivity have been extensively studied in~\cite{Michael2022, frieze2023, Kyriakos2024, Lichev2025}. It was demonstrated that for these properties, a linear edge budget suffices to achieve the desired structure at or shortly after the hitting time. This contrasts sharply with the superlinear hitting time required for achieving minimum degree one, which, as is well known, occurs at $\left(\frac{1}{2} + o(1)\right)n \log n$. 

Much less is known about the construction of small subgraphs, specifically those whose number of vertices remains bounded independently of the size of the host graph. The problem of constructing trees and unicyclic graphs was addressed in~\cite{frieze2023}. Notably, Frieze et al.~\cite{frieze2023} established the following results.

\begin{theorem}\label{tree}
    Let \( k \geq 3 \) be an integer and let \( T \) be a \( k \)-vertex tree. If \( t \geq b \gg \max\{(n/t)^{k-2}, 1\} \) then there exists a \( (t,b) \)-strategy \( B \) of Builder such that
\[
\lim_{n \to \infty} \mathbb{P}(T \subseteq B_t) = 1
\]
and if \( b \ll (n/t)^{k-2} \) then for any \( (t,b) \)-strategy \( B \) of Builder,
\[
\lim_{n \to \infty} \mathbb{P}(T \subseteq B_t) = 0.
\]
\end{theorem}

\begin{theorem}\label{cycle}
    Let \( k \geq 1 \) be an integer and let \( H = C_{2k+1} \) or \( H = C_{2k+2} \). Write \( b^* = b^*(n,t,k) = \max\{n^{k+2}/t^{k+1}, n/\sqrt{t}\} \). If \( t \gg n \) and \( b \gg b^* \) then there exists a \( (t,b) \)-strategy \( B \) of Builder such that
\[
\lim_{n \to \infty} \mathbb{P}(H \subseteq B_t) = 1,
\]
and if \( t \ll n \) or \( b \ll b^* \) then for any \( (t,b) \)-strategy \( B \) of Builder,
\[
\lim_{n \to \infty} \mathbb{P}(H \subseteq B_t) = 0.
\]
\end{theorem}

Theorems \ref{tree} and \ref{cycle} analyze the budget thresholds required for constructing trees
and cycles. Extending these results to any fixed unicyclic graph is straightforward:
once the cycle is formed, the remaining forest can be efficiently built with a
constant budget. Consequently, the smallest graph not covered by their results is the
graph on four vertices and five edges,~$K_4^-$ (the diamond).
A notable challenge is that several natural strategies lead to different upper bounds,
each one optimal within distinct time regimes. A similar issue arises in the case of
constructing cycles.

\subsection{Our results}
In this paper, we discuss optimal strategies for constructing small subgraphs with
multiple cycles. Firstly, we present an optimal strategy for the first unresolved
case outlined in~\cite{frieze2023}, the graph~$K_4^-$.
\begin{theorem}
  \label{thm:k4-}
  Let~$b^* = b^*(n, t, k) = \max\set{n^6 / t^4, n^{4 / 3} / t^{2 / 3}}$.
  If~$t = \omega\paren{n^{6 / 5}}$ and~$b = \omega\paren{b^*}$, then there exists
  a~$(t, b)$-strategy~$B$ of Builder such that
  \begin{equation*}
    \lim_{n \to \infty} \Prob{K_4^- \subset B_t} = 1,
  \end{equation*}
  and if~$t = o\paren{n^{6 / 5}}$ or~$b = o\paren{b^*}$ then for
  any~$(t, b)$-strategy~$B$ of Builder,
  \begin{equation*}
    \lim_{n \to \infty} \Prob{K_4^- \subset B_t} = 0.
  \end{equation*}
\end{theorem}
We also consider another kind of multi-cyclic graph. For an integer~$k \geq 1$,
let~$T_k$ be the graph consisting on~$k$ triangles intersecting in a single
vertex. Then we have the following result.
\begin{theorem}
  \label{thm:k-butterfly}
  Let~$k \geq 1$ be an integer, and
  let~$b^* = b^*(n, t, k) = \max\set{n^{4k - 1} / t^{3k - 1}, n / \sqrt{t}}$.
  If~$t = \omega\paren{n^{\frac{4k - 1}{3k}}}$ and~$b = \omega\paren{b^*}$, then there exists
  a~$(t, b)$-strategy~$B$ of Builder such that
  \begin{equation*}
    \lim_{n \to \infty} \Prob{T_k \subset B_t} = 1,
  \end{equation*}
  and if~$t = o\paren{n^{\frac{4k - 1}{3k}}}$ or~$b = o\paren{b^*}$ then for
  any~$(t, b)$-strategy~$B$ of Builder,
  \begin{equation*}
    \lim_{n \to \infty} \Prob{T_k \subset B_t} = 0.
  \end{equation*}
\end{theorem}
We present in Figures \ref{fig:k4-} and \ref{fig:T-k} visualizations of
the dependencies between the time~$t$ and minimum threshold for a budget~$b$ in terms
of $n$. For better presentation, we scale them logarithmically. Also, notice that for
the case $k=1$, $T_k$ is a triangle and our formulas generalize the result in~\cite{frieze2023} for
a cycle of length 3.

The paper is organized as follows. In Section~2, we present the necessary notation and tools for our proof. Section~3 focuses on the optimal strategy for $K_4^-$. We will divide the argument into two parts to prove the 0-statement and 1-statement of Theorem~\ref{thm:k4-}. Section~4 contains the proof of Theorem \ref{thm:k-butterfly} which is also divided into two parts, one for the 0-statement and the other for the 1-statement. In the final section, Section~5, we will discuss extensions of our approach and several pertinent open problems.

\begin{figure*}[t!]
  \captionsetup{width=0.879\textwidth,font=small}
  \centering
  \begin{tikzpicture}
    \begin{axis}[
        xmin=1, xmax=2,
        ymin=0, ymax=1.5,
        x=10cm,y=5cm,
        samples=100,
        axis y line=center,
        axis x line=middle,
        xlabel={$\log_n{t}$},
        xtick={6/5,7/5,2},
        extra y ticks={0},
        xticklabels={6/5,7/5,2},
        ylabel={$\log_n{b}$},
        ytick={6/5,1,2/5,0},
        yticklabels={6/5,1,2/5,0},
        every axis x label/.style={
            at={(ticklabel* cs:1.05)},
            anchor=west,
        },
        every axis y label/.style={
            at={(ticklabel* cs:1.05)},
            anchor=south,
        },
        reverse legend,
    ]
        % Guide
        \addplot[color=gray,loosely dotted,thick,domain=1:2] {4/3-2*x/3};

        \addplot[color=gray,loosely dotted,thick,domain=1:6/5] {6/5};

        % K_4^-
        \addplot[color=red,thick,domain=6/5:7/5] {6-4*x};
        \addplot[color=red,thick,domain=7/5:2] {4/3-2*x/3};

        \legend{$x\mapsto 4/3-2x/3$,,$K_4^-$}
    \end{axis}
  \end{tikzpicture}
  \caption{Budget threshold for $K_4^-$.}
  \label{fig:k4-}
\end{figure*}
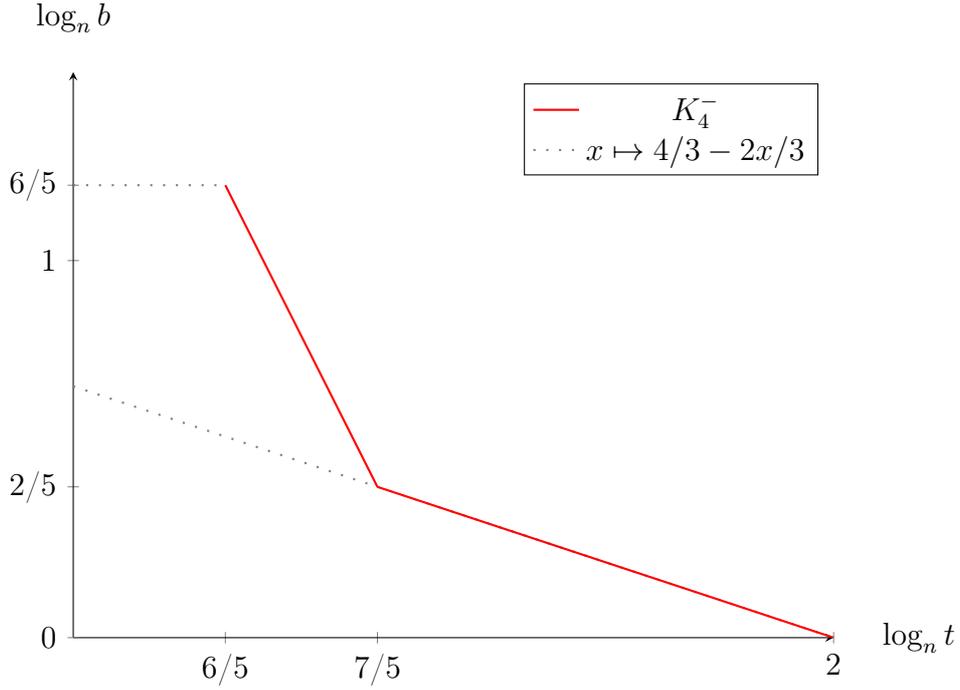

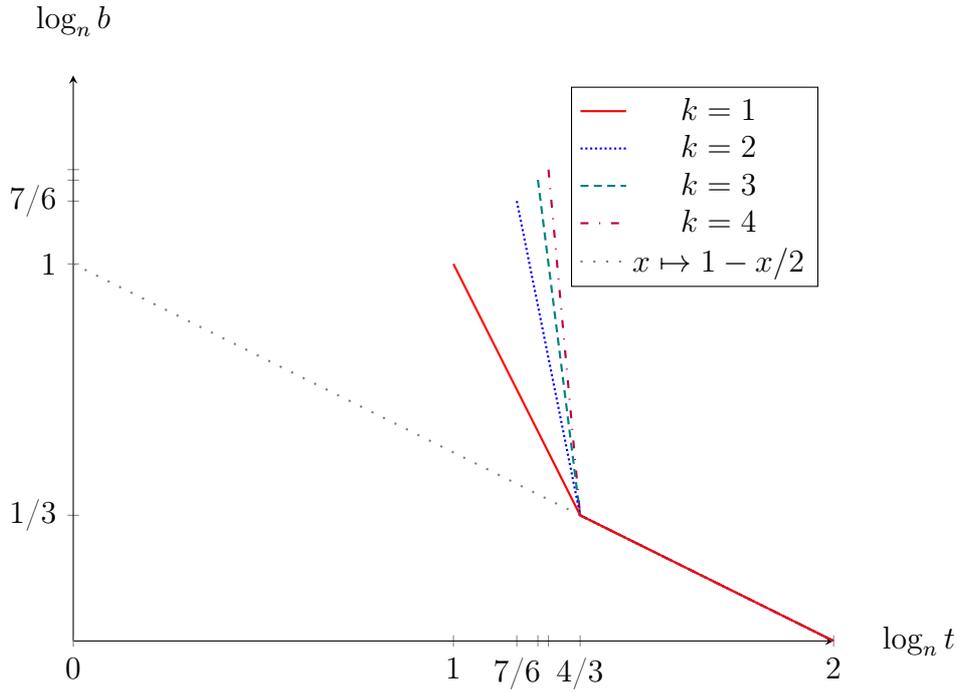
\begin{figure*}[t]
  \captionsetup{width=0.879\textwidth,font=small}
  \centering
  \begin{tikzpicture}
    \begin{axis}[
        xmin=0, xmax=2,
        ymin=0, ymax=1.5,
        x=5cm,y=5cm,
        samples=100,
        axis y line=center,
        axis x line=middle,
        xlabel={$\log_n{t}$},
        xtick={0,1,7/6,11/9,5/4,4/3,2},
        extra x ticks={0},
        xticklabels={0,1,7/6,,,4/3,2},
        ylabel={$\log_n{b}$},
        ytick={7/6,11/9,5/4,1,1/3,0},
        yticklabels={7/6,,,1,1/3,0},
        every axis x label/.style={
            at={(ticklabel* cs:1.05)},
            anchor=west,
        },
        every axis y label/.style={
            at={(ticklabel* cs:1.05)},
            anchor=south,
        },
        reverse legend,
    ]
        % Guide
        \addplot[color=gray,loosely dotted,thick,domain=0:2] {1-x/2};

        % k=4
        \addplot[color=purple,loosely dashdotted,thick,domain=5/4:4/3] {15-11*x};
        \addplot[color=purple,loosely dashdotted,thick,domain=4/3:2] {1-x/2};

        % k=3
        \addplot[color=teal,densely dashed,thick,domain=11/9:4/3] {11-8*x};
        \addplot[color=teal,densely dashed,thick,domain=4/3:2] {1-x/2};

        % k=2
        \addplot[color=blue,densely dotted,thick,domain=7/6:4/3] {7-5*x};
        \addplot[color=blue,densely dotted,thick,domain=4/3:2] {1-x/2};

        % k=1
        \addplot[color=red,thick,domain=1:4/3] {3-2*x};
        \addplot[color=red,thick,domain=4/3:2] {1-x/2};

        \legend{$x\mapsto 1-x/2$,,$k=4$,,$k=3$,,$k=2$,,$k=1$}
    \end{axis}
  \end{tikzpicture}
  \caption{Budget thresholds for $T_k$.}
  \label{fig:T-k}
\end{figure*}

\section{Notation and Tools}

This section contains some lemmas that will be useful further in the proofs of the
main theorems.

We first make precise our notation. For a graph~$G$, we denote its vertex set
by~$V(G)$, its edge set by~$E(G)$, and we let~$e(G) = |E(G)|$. We usually assume
that, if~$G$ is a host graph with~$n$ vertices, then~$V(G) = [n]$. For a set of
vertices~$S$, the \defin{neighborhood} of~$S$ in~$G$ is the set of vertices
from~$V(G) \setminus S$ which are adjacent to a vertex from~$S$, and we denote this set
by~$N_G(S)$. When~$S = \set{v}$, we also write~$N_G(v)$. The \defin{degree} of a
vertex~$v$ in~$G$ is~$d_G(v) = |N_G(v)|$. We denote the path and the cycle on~$n$
vertices by~$P_n$ and~$C_n$ respectively, and a disjoint set of~$k$ edges
by~$kK_2$. We also denote by~$K_4^{-}$ the graph on four vertices and five edges,
and~$K_3^+$ as the graph on four vertices that contains a triangle and one additional edge.

Given two real valued functions~$f$ and~$g$,
if~$\lim_{n \to \infty} f(n) / g(n) = 0$, then we equivalently say
that~$f = o(g)$,~$f \ll g$,~$g = \omega(f)$, and~$g \gg f$.
If~$\lim_{n \to \infty} f(n) / g(n) < \infty$, then we write~$f = O(g)$,
and~$g = \Omega(f)$. If~$f = O(g)$ and~$f = \Omega(g)$, we then
write~$f = \Theta(g)$. We say that an event~$E = E(n)$ occurs \defin{with high
  probability}, or \whp for short if~$\Prob{E} = 1 - o(1)$. We will make some abuse
of notation when divisibility problems arise. For instance, we may say that a set has
size~$t / 3$, even if~$t$ is not divisible by~$3$.

Denote the graph generated by random process at time~$t$ by~$G_t$. The \defin{hitting time} for a
monotone graph property~$\cP$ is the minimum time~$t$ for which~$G_t$ has the
property~$\cP$. Denote by~$B_t$ the spanning subgraph of~$G_t$ that consists of the
edges that are purchased by Builder by time~$t$. A~\defin{$(t, b)$-strategy} of the
builder is a function that, for any~$1 \leq s \leq t$, given~$B_{s - 1}$ decides
whether to purchase the edge presented at time~$s$, under the limitation that~$B_s$
has at most~$b$ edges.

In the rest of the paper, we would use the following concentration inequalities.

\begin{lemma}[Chernoff bound]
  \label{lem:chernoff}
  If~$X$ is a binomial random variable, and~$0 < \eps < 1$, then
  \begin{equation*}
    \Prob{|X - \Expect X| \geq \eps \Expect X} \leq 2 \cdot \exp\paren{- \frac{\eps^2
        \Expect X}{3}}.
  \end{equation*}
\end{lemma}

\begin{lemma}[Bennett's inequality]
  \label{lem:bennett_inequality}
  Let~$X_1, \dots, X_T$ be independent random variables defined on a finite probability
  space, and having finite variances. Suppose that~$|X_i| \leq b$, for all
  $1 \leq i\leq T$ and some~$b > 0$.
  Let~$\sigma^2 = \sum_{i = 1}^T \Expect{(X_i - \Expect X_i)^2}$. Then, for every~$\ell \geq 0$,
  \begin{equation*}
    \Prob{\left|\sum_{i = 1}^T X_i - \Expect{X_i}\right| \geq \ell} \leq
    2\cdot \exp\paren{-\frac{\sigma^2}{b^2}h\paren{\frac{\ell b}{\sigma^2}}},
  \end{equation*}
  where~$h(x) = \paren{1 + x}\log \paren{1 + x} - x$, for~$x \geq 0$.
\end{lemma}

We will use the following two well known lemmas to prove an auxiliary lemma about the
containment of two particular small graphs in a subgraph of the random graph~$G_t$.

\begin{lemma}
  \label{lem:model_switching}
  Let~$\cP$ be an increasing graph property. Suppose that~$f(n) \to \infty$
  and~$\sqrt{p} n \to \infty$ as~$n \to \infty$.
  If~$T = T(n) \geq p \binom{n}{2} + f(n)\sqrt{p} n$,
  then~$\Prob{G_T \in \cP} \geq \Prob{G(n, p) \in \cP} - o(1)$.
\end{lemma}

\begin{lemma}
  \label{lem:threshold_graph_containment}
  Let~$G \sim G(n, p)$, with~$p = p(n)$, and let~$k \geq 1$ be an integer. Then the
  following statements hold \whp
  \begin{enumerate}[(i)]
  \item If~$p = \omega\paren{n^{-3/2}}$, then~$P_3 \subset G$.
  \item If~$p = \omega\paren{n^{-2}}$, then~$kK_2 \subset G$.
  \end{enumerate}
\end{lemma}

Based on the previous two lemmas we can establish the threshold for $P_3$ and $kK_2$ in $G_T$.

\begin{lemma}
  \label{lem:threshold_graph_Gnm}
  Let~$G \sim G_T$ be a random graph with~$n$ vertices and~$T$ edges, and
  let~$N \subset V(G)$ be a set of vertices of~$G$. Let~$k \geq 1$ be an integer. Then,
  \begin{enumerate}[(i)]
  \item If~$\paren{T / n^2} \cdot |N|^{3 / 2} = \omega\paren{1}$, then \whp~$P_3
      \subset G[N]$.
  \item If~$\paren{T / n^2} \cdot |N|^2 = \omega\paren{1}$, then \whp~$kK_2 \subset G[N]$.
  \end{enumerate}
\end{lemma}
\begin{proof}
  With~$p = T / n^2$, we have that~$T \geq 2p\binom{n}{2}$. Observe that the two stated
  properties are increasing. Hence, by \Cref{lem:model_switching}, we have that
  \begin{equation*}
    \Prob{P_3 \subset G[N]} \geq \Prob{P_3 \subset G(n, p)[N]} - o(1).
  \end{equation*}
  Clearly, the graph~$H = G[N]$ has distribution~$G(|N|, p)$. Then, by
  \Cref{lem:threshold_graph_containment},
  if~$p = T / n^2 = \omega\paren{|N|^{-3/2}}$, then \whp~$P_3 \subset H$, and
  therefore~$P_3 \subset G[N]$.

  Similarly, if~$p = T / n^2 = \omega{(|N|^{-2})}$, then \whp~$kK_2 \subset H$, and
  therefore~$kK_2 \subset G[N]$.
\end{proof}

Finally, we establish likely properties arising from random processes.

\begin{lemma}
  \label{lem:degree_concentration}
  Let~$T \geq 1$ be an integer, and let~$C > 1$. Then \whp every vertex~$v$ of
  the~$n$-vertex random graph with~$T$ edges~$G_T$
  satisfies~$T / Cn \leq d_{G_T}(v) \leq CT / n$.
\end{lemma}

\begin{lemma}
  \label{claim:sample_T_at_most_5_times}
  Consider~$T = O(n^{7 / 5})$ uniform and independent samples of elements from the
  set~$\binom{[n]}{2}$. Then, \whp every member of~$\binom{[n]}{2}$ was sampled at
  most~$5$ times.
\end{lemma}
\begin{proof}
  Let~$e_1, \dots, e_T$ be the random sampled pairs. For a fixed
  pair~$e \in \binom{[n]}{2}$, we have that, for all~$i = 1, \dots,
  T$,~$\Prob{e = e_i} = 1 / \binom{n}{2} = O\paren{n^{-2}}$. Clearly these events are
  independent. Thus, the probability that there is a set~$I \subset [T]$ of
  size~$|I| \geq 6$ such that~$e = e_j$ for all~$j \in I$ is at
  most~$\binom{T}{6} \Prob{e = e_1}^6 = O\paren{T^6 / n^{12}} = O\paren{n^{-18 /
      5}}$, where we have used the fact that~$T = O\paren{n^{7 / 5}}$. Therefore,
  taking a union bound over all elements in~$\binom{[n]}{2}$, the probability that
  some pair~$e \in \binom{[n]}{2}$ satisfies the above property is at
  most~$n^2 \cdot O(n^{-18 / 5}) = O(n^{-8 / 5}) = o(1)$.
\end{proof}

\begin{claim}
  \label{claim:every_edge_in_at_most_c_Ns}
  For any~$\eps > 0$, there exists~$c_\eps$ such that, if~$t = O\paren{n^{3 / 2} - \eps}$, then \whp the random graph~$G_t$ satisfies that for every set~$v_i$ of its vertices, every pair of vertices of the graph is contained in the neighborhood of at most~$c_\eps$ of the vertices~$v_i$.
\end{claim}
\begin{claimproof}
  Let~$s$ be the maximum degree of~$G_t$. We have that \whp $s = O(t / n)$. Choose $c_\eps$ large enough so that $\eps(c_\eps + 1) \geq 3$. For a
  pair~$e' \in \binom{[n] \setminus \set{v_i}}{2}$,
  \begin{equation*}
    \Prob{e' \subset N_{G^t}(v_i)} = \frac{\binom{s}{2}}{\binom{n - 1}{2}} \leq
    \frac{4s^2}{n^2} \leq O\paren{\frac{t^2}{n^4}}.
  \end{equation*}
  Also, given that any pair of vertices that contains~$v_i$ cannot belong
  to~$N_{G^t}(v_i)$, we have that for every pair~$e \in
  \binom{[n]}{2}$,~$\Prob{e \subset N_{G^t}(v_i)} = O(t^2 / n^{4})$. Note also that,
  for~$j \neq i$, the events~$e \in N_{G^t}(v_i)$ and~$e \in N_{G^t}(v_j)$ are
  independent. By taking the union bound over all choices of $c_\eps + 1$ vertices, the probability that there is a set~$I \subset [n]$
  with~$|I| \geq c_\eps + 1$ such that~$e \subset N_{G^t}(v_j)$ for all~$j \in I$ is at most
  \begin{equation*}
    \binom{n}{c_\eps + 1} \Prob{e \in N_1}^{c_\eps + 1} \leq O\paren{\frac{n^{c_\eps + 1} t^{2(c_\eps + 1)}}{n^{4(c_\eps + 1)}}} =
    O\paren{\frac{n^{c_\eps + 1} n^{3(c_\eps + 1) - \eps(c_\eps + 1)}}{n^{4(c_\eps + 1)}}} \leq o(n^{-3}),
  \end{equation*}
  where we have made a substitution for $t$ and the inequality for $c_\eps$.  Therefore, taking a union bound over all edges
  in~$\binom{[n]}{2}$, yields the probability that some pair~$e \in \binom{[n]}{2}$ is
  contained in at least~$c_\eps + 1$ sets~$N_{G^t}(v_j)$ is at
  most~$n^2 \cdot o(n^{-3}) = o(n^{-1}) = o(1)$.
\end{claimproof}

\section{Proof of Theorem~\ref{thm:k4-}}

\subsection{\texorpdfstring{$1$}{1}-statement}

Let~$t \gg n^{6 / 5}$, and let~$b \gg b^*$,
where~$b^* = \max\set{n^{6} / t^{4}, n^{4 / 3} / t^{2 / 3}}$.

\noindent \textbf{Short time with large budget.} We first consider the case
where~$n^{7 / 6} \ll t = O\paren{n^{7 / 5}}$. In this case we have
that~$b \gg b^* = n^6 / t^4$.  Let~$T = t / 3$, and let~$r = \omega\paren{n^7 / T^5}$ such
that~$rt / n = o(b)$.

Roughly speaking, the strategy in this case is to select a set of~$r$ vertices and
then reveal the first~$T$ edges of the graph, buying every edge that intersects these
selected vertices as long as the neighborhoods have size at most~$\Theta(t / n)$. Thus, we
obtain a set of neighborhoods which cover a large number of edges of the graph,
having purchased at most~$o(b)$ edges. We then reveal the next~$T$ edges and buy at
most~$b / 2$ edges contained in these neighborhoods,
obtaining~$\omega\paren{n^2 / T}$ triangles. In the last step, we reveal the
remaining~$T$ edges, at least one of which will, with high probability, complete the
graph~$K_4^-$. We now formalize this strategy.

Let~$R$ be an arbitrary set of~$r$ vertices of the graph~$G_0$. The strategy in this
case is to reveal the first~$T$ edges and buy every edge that
intersects~$R$. Let~$G^R$ be the graph obtained as a result of this step. By
\Cref{lem:degree_concentration}, \whp every vertex~$v \in R$
satisfies~$T / 3n \leq d_{G_T}(v) \leq 3T / n$. Hence, the graph~$G^R$ \whp contains
at most~$r \cdot 3T / n = o(b)$ edges. We also have that~$d_{G^R}(v) \geq T / 3n$. We
now make the following claim.

\begin{claim}
  \label{claim:every_edge_in_at_most_10_Ns}
  With high probability, every pair~$e \in \binom{[n]}{2}$ is contained in at
  most~$10$ of the sets~$N_{G^R}(v_i)$.
\end{claim}
\begin{claimproof}
  Fix~$i = 1, \dots, r$. Observe that, for a given
  set~$Q \subset [n] \setminus \set{v_i}$ of size~$k$ we have
  that~$\Prob{N_{G^R}(v_i) = Q}$ is the same for all choices of~$Q$ of size $k$. Then,
  if~$s = |N_{G^R}(v_i)| = O (T / n) = O(t / n)$, for a
  pair~$e' \in \binom{[n] \setminus \set{v_i}}{2}$,
  \begin{equation*}
    \Prob{e' \subset N_{G^R}(v_i)} = \frac{\binom{s}{2}}{\binom{n - 1}{2}} \leq
    \frac{4s^2}{n^2} \leq O\paren{\frac{t^2}{n^4}}.
  \end{equation*}
  Also, given that any pair of vertices that contains~$v_i$ cannot belong
  to~$N_{G^R}(v_i)$, we have that for every pair~$e \in
  \binom{[n]}{2}$,~$\Prob{e \subset N_{G^R}(v_i)} = O(t^2 / n^{4})$. Note also that,
  for~$j \neq i$, the events~$e \in N_{G^R}(v_i)$ and~$e \in N_{G^R}(v_j)$ are
  independent. Thus, the probability that there is a set~$I \subset [r]$
  with~$|I| \geq 11$ such that~$e \subset N_{G^R}(v_j)$ for all~$j \in I$ is at most
  \begin{equation*}
    \binom{r}{11} \Prob{e \in N_1}^{11} \leq O\paren{\frac{r^{11} t^{22}}{n^{44}}} \leq
    o\paren{\frac{b^{11}t^{11}}{n^{33}}} \leq o(n^{-11 / 5}),
  \end{equation*}
  where we have used the fact that~$r = o\paren{bn / t}$,
  and~$b, t = O\paren{n^{7 / 5}}$.  Therefore, taking a union bound over all elements
  in~$\binom{[n]}{2}$, the probability that some pair~$e \in \binom{[n]}{2}$ is
  contained in at least~$11$ sets~$N_{G^R}(v_j)$ is at
  most~$n^2 \cdot o(n^{-11 / 5}) = o(n^{-1/5}) = o(1)$.
\end{claimproof}

We proceed by revealing the next~$T$ edges and buy every edge contained in at least
one set~$N_{G^R}(v_i)$, as long as not more than~$b / 2$ edges are bought during this
step. Let~$G^T$ be the graph obtained as a result of this process.
Thus,~$G^R \subset G^T$.

Let~$X_1, \dots, X_T$ be independent copies of a random variable that, given a uniformly
at random pair~$e$ from~$\binom{[n]}{2}$, it counts the number of vertices~$v \in R$
such that~$e \subset N_{G^R}(v)$, and let~$X = X_1 + \cdots + X_T$. We have
that~$\Expect{X_1} = \Theta\paren{r\cdot (T / 3n)^2 / n^2} = \Theta\paren{rT / n^4} = \omega \paren{n^3
  / T^3}$.
Hence,~$\mu = \Expect{X} = T \cdot \Expect{X_1} = \omega\paren{n^3 / T^2}$. We claim that \whp
~$X \geq \mu / 2$. Indeed, by \Cref{claim:every_edge_in_at_most_10_Ns}~$X_i$ \whp is at
most~$10$, and~$\Prob{X_i \geq 1} \leq \Expect{X_i} = o(1)$. Therefore, we have
that~$\sigma^2 = \sum_{i = 1}^T \Expect{(X_i - \Expect X_i)^2} \leq 100\mu (1 + o(1))$. Hence, by using
\Cref{lem:bennett_inequality} with~$\ell = \mu / 2$, we have
that~$\Prob{X \geq \mu / 2} = o(1)$.

By using \Cref{lem:bennett_inequality}, if~$E \subset \binom{[n]}{2}$ be a set of~$T$ pairs
of vertices, and if, for a randomly with uniform probability chosen
pair~$e \in \binom{[n]}{2}$,~$Y_1, \dots, Y_T$ are independent copies of a random
variable defined
as~$\mathbbm{1}_{\set{e \in E}} \cdot |\set{v \in R : e \subset N_{G^R}(v)}|$, then,
\whp~$Y_1 + \cdots + Y_T = o(\mu)$. Let~$e_1, \dots, e_T$ be the random edges revealed in the
second step of the construction to obtain~$G^T$. For each~$e_i$, let~$Z_i$ be the
random variable that counts the number of vertices~$v \in R$ such
that~$e_i \subset N_{G^R}(v)$. Note that, in the case that one of these edges,~$e_j$ is
such that~$Z_j \geq 2$, then~$K_4^- \subset G^T$. Therefore, we assume
that~$Z_i \in \set{0, 1}$. We define~$a = e(G^T) - e(G^R)$. By
\Cref{claim:sample_T_at_most_5_times}, \whp every pair can be sampled at most~$5$
times. Then, using \Cref{claim:every_edge_in_at_most_10_Ns}, in the case
that~$a \leq b / 2$, i.e., we were not limited by the budget during this step, we can
see that each pair of vertices is \whp overcounted at most~$5 \times 10 = 50$ times
in~$a$.
Thus,~$a = \sum_{i = 1}^T Z_i \geq \frac{1}{50} \paren{\sum_{i = 1}^T X_i} - \sum_{i = 1}^T Y_i \geq
\mu / 50 - o(\mu) \geq \mu / 100 = \omega\paren{n^3 / T^2}$. In the case
that~$a = b / 2$, i.e., we purchased all~$b / 2$ edges during this step, we have
that~$a = \omega\paren{n^6 / t^4} = \omega\paren{n^3 / T^2}$, given
that~$t \leq O\paren{n^{7 / 5}} = o\paren{n^{3 / 2}}$. In both cases we
have~$a = \omega\paren{n^3 / T^2}$.

Observe that, for every edge~$\set{x, y} \in E\paren{G^T} \setminus E\paren{G^R}$,
if~$v$ is the unique vertex from~$R$ such that~$\{x, y\} \subset N_{G^R}(v)$, then for every
vertex~$z \in N_{G^R}(v) \setminus \set{x, y}$, the graphs~$G^T + xz$ and~$G^T + yz$
contain~$K_4^-$ as a subgraph. Let~$C$ be the set of all pairs~$\set{x, z}$
and~$\set{y, z}$ with~$\set{x, y} \in E\paren{G^T} \setminus E\paren{G^R}$
and~$z \in N_{G^R}(v) \setminus \set{x, y}$, where~$\{x, y\} \subset N_{G^R}(v)$. Then,
by~\Cref{claim:every_edge_in_at_most_10_Ns}, we have
that~$|C| \geq \paren{2a / 10} \cdot \paren{T / 3n - 2} \geq aT / 5n = \omega\paren{n^2 / T}$.

We finish the strategy by revealing the last~$T$ edges and buying every edge that
lies in~$C$ to obtain the graph~$B_t$. Before the second step, it is possible that
some elements of~$C$ have been revealed when considering the first~$T$
edges. However, no element of~$C$ was revealed during the second step. That is, the
sets~$E\paren{G^T} \setminus E\paren{G^R}$ and~$C$ are disjoint.  Let~$C' \subset C$ be the set of
pairs of~$C$ that were not revealed with the first~$T$ edges. Note that
\whp~$|C'| \geq |C| - 3T/n^2 \geq |C| / 2$. We also have
that~$\Prob{E\paren{B_t} \cap C' \neq \emptyset} \geq \Prob{E\paren{G_T} \cap C' \neq
  \emptyset}$.

Hence, if we are able to show that, \whp the random graph~$G_T$ contains an edge
from~$C'$, we can conclude that \whp~$K_4^- \subset B_t$, and the strategy is
successful in this case.

For a uniformly at random pair~$e$ from~$\binom{[n]}{2}$, let~$X_1, \dots, X_T$ be
independent copies of an indicator random variable defined
as~$\mathbbm{1}_{\set{e \in C'}}$, and let~$X = \sum_{i = 1}^T X_i$. Note that,
for~$e$ as
above,~$\Prob{e \in C'} = |C'| / \binom{n}{2} = \omega\paren{n^2 / T} \cdot n^{-2} =
\omega\paren{1 / T}$. Hence, we have that~$\Expect X = \omega(1)$. Also, by
\Cref{lem:chernoff}, we have
that~$\Prob{X \leq \Expect X / 2} \leq \Prob{|X - \Expect X| \geq \Expect X / 2} \leq
2 \cdot \e^{-\Expect X / 12} = e^{-\omega(1)} = o(1)$. Thus,
\whp~$X > \Expect X / 2 = \omega(1)$. Therefore, \whp algorithm created a $K_4^-$.

\noindent \textbf{Long time with short budget.} We now assume
that~$t = \omega\paren{n^{7 / 5}}$. For this case we have
that~$b \gg b^* = n^{4 / 3} / t^{2 / 3}$. Let~$T = t / 2$.

Similar to the previous case, we now fix a vertex and, while revealing the first~$T$
edges, we buy each edge incident to it, obtaining a neighborhood of size~$O(T / n)$
and having purchased at most~$b / 2$ edges. We then reveal the remaining~$T$ edges,
and buy every edge contained in the neighborhood; if with
high probability two of them are non-disjoint, we construct a copy of~$K_4^-$. We now formalize this strategy.

Let~$v$ be a vertex of the graph~$G_0$. We start by revealing the first~$T$ edges and
buy every edge incident to~$v$, as long as we buy at most~$b / 2$ edges. Let~$G^v$ be
the graph obtained in this step. Let~$N = N_{G^v}(v)$. By
\Cref{lem:degree_concentration}, \whp~$d_{G_T}(v) \geq T / 3n$. Therefore, we
distinguish the cases where~$|N| \geq T / 3n = \omega\paren{n^{2 / 5}}$,
and~$|N| = b / 2 = \omega\paren{n^{4 / 3} / T^{2 / 3}}$, i.e., where we were and were not
restricted by the budget respectively. We proceed to reveal the remaining~$T$ edges
and buy each contained in~$N$, given that we only buy at most~$b / 2$ of them,
obtaining the graph~$B_t$. Note
that~$\Prob{P_3 \subset B_t[N]} = \Prob{P_3 \subset G_T[N]}$.

Note that~$T / n^2 \gg n^{7 / 5} / n^{2} = n^{-3 / 5}$. In the case
where~$|N| = \omega\paren{n^{2 / 5}}$ we have
that~$\paren{T / n^2} |N|^{3 / 2} \gg n^{-3 / 5} \paren{n^{2 / 5}}^{3 / 2} = 1$. In the
case where~$|N| = \omega\paren{n^{4 / 3} / T^{2 / 3}}$, we have
that~$\paren{T / n^2} |N|^{3 / 2} \gg \frac{T}{n^2} \paren{\frac{n^{4 / 3}}{T^{2 /
      3}}}^{3 / 2} = \Omega(1)$. In either case, we have
that~$\paren{T / n^2} |N|^{3 / 2} = \omega(1)$. Thus, by \Cref{lem:threshold_graph_Gnm}, we
conclude that $P_3 \subset G_T[N]$, and the strategy succeeds. This completes the proof of
the~$1$-statement.

\subsection{\texorpdfstring{$0$}{0}-statement}

If~$t \ll n^{6 / 5}$, the graph~$G_t$ \whp does not contain the
graph~$K_4^-$. Therefore, suppose that~$t \gg n^{6 / 5}$, and that~$b \ll b^*$,
where~$b^* = \max\set{n^{6} / t^{4}, n^{4 / 3} / t^{2 / 3}}$. As in
the~$1$-statement, we will split the analysis into two cases depending on~$t$.

\noindent \textbf{Short time with large budget.} Let us first consider the case
where~$n^{6 / 5} \ll t \ll n^{7 / 5}$. In this case we have
that~$b \ll b^* = n^6 / t^4$. We will now prove that there is no~$(t, b)$-strategy
that generates the graph~$K_4^-$.

Clearly, we may assume that~$b = \omega\paren{t / n}$, as increasing the value of~$b$
makes the claim stronger. Let~$G \subset G_t$ be a spanning subgraph with~$b$ edges of the
random graph~$G_t$. We will first prove that, \whp the number of copies of~$C_4$
and~$K_3^+$ in~$G$ is~$O(bt^3 / n^4)$. Assuming this, note that these are the only
graphs that can be formed by taking the graph~$K_4^-$ and removing one edge from
it. Then,~$G$ contains~$O(bt^3 / n^4)$ candidate pairs of vertices that generate the
graph~$K_4^-$, and the probability of hitting one of these edges by revealing only
one edge is at most~$O\paren{bt^3 / n^6} = o\paren{1 / t}$. Hence, by
revealing~$O(t)$ additional edges, \whp none of these edges hit any of the candidate
pairs, and the claim follows. Thus, it only remains to show that \whp the number
copies of~$C_4$ and~$K_3^+$ in~$G$ is~$O(bt^3 / n^4)$.

We start by upper-bounding the number of copies of the graph~$C_4$ in~$G$. We start
by noticing that there are at most~$O(bn / t) \cdot O\paren{t^2 / n^2} = O(bt / n)$
copies of~$P_3$ in~$G$. We can only generate a copy of~$P_4$ by extending a copy
of~$P_3$ or by joining two disjoint edges with a third one. We analyze these two
cases separately. For a vertex~$v$, let~$Y_v$ be the random variable that counts the
number of copies of~$P_3$ with end-vertex~$v$. We have
that~$\Expect{Y_v} = \Theta\paren{\paren{bt / n} / n} = \Theta\paren{bt / n^2}$.
Let~$X_1, \dots, X_t$ be independent copies of a random variable that, when revealing a
random edge on top of~$G$ (possibly an existing one), counts the number of new copies
of the graph~$P_4$ formed by extending a copy of~$P_3$. Given that any copy of~$P_3$ can be extended to a copy of~$P_4$
by adding an edge to either of its end-vertices, we
have~$\mu = \Expect{X_i} = \Theta\paren{\Expect{Y_v}}$. We will use
\Cref{lem:bennett_inequality} to bound the number of copies of~$P_4$.
Let~$S_t = \sum_{i = 1}^t X_i - \mu$. Then, given that every vertex is an end-vertex of
at most~$O(t^2 / n^2)$ copies of~$P_3$, we have that, for
all~$i$,~$|X_i - \mu| \leq ct^2 / n^2$, for some constant~$c$.  Let~$a = ct^2 /
n^2$. Given that, \whp the random variable~$X_i$ is bounded from both sides, we have
that~$\sigma^2 = \sum_{i = 1}^t \Expect\paren{X_i - \mu}^2 \leq at \mu$. We now consider two
cases depending on~$\sigma^2$.

Let~$h(x) = (1+x)\log(1+x)-x$. In the case
where~$\sqrt{a} \cdot t\mu \leq \sigma^2 \leq at\mu$, we have that, for a large enough constant
$C$, we have $h(act\mu / \sigma^2) \geq 1$. Therefore, we have
that
\begin{align*}
  \Prob{S_t > ct\mu} &\leq 2 \exp\paren{-\frac{\sigma^2}{a^2}\cdot h\paren{\frac{act\mu}{\sigma^2}}} \leq
                     \exp\paren{-\frac{\sigma^2}{a^2}} \leq
                     \exp\paren{-\frac{\sqrt{a}t\mu}{a^2}} \\
                   & = \exp\paren{-\frac{t\mu}{t^3 / n^3}} = \exp\paren{-\frac{tc\cdot bt /
                     n^2}{t^3 / n^3}} \\
  &= \exp\paren{-\frac{cbn}{t}} \leq
    \exp\paren{-\omega(1)} = o(1),
\end{align*}
where the last inequality holds by our choice of $b$. Suppose now
that~$\sigma^2 \leq \sqrt{a}\cdot t\mu$. Given
that~$a = \omega\paren{1}$, we
have~$h\paren{at\mu /\sigma^2} \geq at\mu /\sigma^2$. Therefore,
\begin{align*}
  \Prob{S_t > t\mu} &\leq 2\exp\paren{-\frac{\sigma^2}{a^2}\cdot h\paren{\frac{at\mu}{\sigma^2}}} \leq
                       \exp\paren{-\frac{\sigma^2}{a^2}\cdot \frac{at\mu}{\sigma^2}} \\
                     & = \exp\paren{-\frac{t\mu}{a}} \leq \exp\paren{-\frac{bt^2 / n^2}{ct^2 / n^2}}  = o(1).
\end{align*}
Hence, \whp we have that~$S_t < c_4t\mu$, for a large enough constant~$c_4$.
Let~$q = \sum_{i=1}^t X_i$. Given that~$c_4t\mu > S_t = q - t\mu$, we have
that~$q < (c_4 + 1)t\mu$. Therefore, there are at
most~$(c_4 + 1)t\mu = (c_4 + 1)\cdot bt^2 / n^2$ copies of~$P_4$ that extend a copy
of~$P_3$. We now consider copies of~$P_4$ created by joining two disjoint edges with
a third edge. The expectation of generating a copy of~$P_4$ by joining two edges can
be upper bounded by the square of the number of endpoints of the present
edges. Similarly as in the previous case, let~$Y_1, \dots, Y_t$ be independent copies
of a random variable that, when revealing a random edge on top of~$G$, counts the
number of new copies of the graph~$P_4$ obtained by joining two previously disjoint
edges. We have
that~$\mu_2 = \Expect{Y_i} = \Theta\paren{b^2 / n^2} = o\paren{bt / n^2}$. Note
that~$\mu_2 = o(\mu)$. Also, we can bound the differences~$|Y_i - \mu_2|$ and the variance
as in the previous case. Using~\Cref{lem:bennett_inequality}, we get
that~$\sum_{i=1}^tY_i = O\paren{t\mu}$, which is negligible compared to the previous
case. We conclude that~$G$ contains at most~$O(bt^2 / n^2)$ copies of~$P_4$. Finally,
we let~$Z_1, \dots, Z_t$ be independent copies of the same random variable that, when
revealing a random edge on top of~$G$, counts the number of new copies of the
graph~$C_4$. Note that each copy of~$P_4$ generates exactly one possible candidate
pair of vertices for generating a copy of~$C_4$.
Therefore,~$\mu_3 = \Expect{Z_i} \leq c_5bt^2 / n^2 \cdot 1 / n^2 = c_5 bt^2 / n^4$. Observe
that \whp any pair of distinct vertices are end-vertices of~$O(t^3 / n^4)$ copies
of~$P_4$. Indeed, the neighborhoods of such vertices \whp have size~$O(t / n)$. From
the at most~$O(t^2 / n^2)$ pairs of vertices between both neighborhoods,
only~$O(t^3 / n^4)$ of them are edges in~$G_t$. By using
\Cref{lem:bennett_inequality}, we bound~$|X_i - \mu_3| \leq c_6 t^3 / n^4 = a$. We also
have that~$\sigma^2 = \sum_{i = 1}^t \Expect\paren{Z_i - \mu_3}^2 \leq at\mu_3$, and similarly as
above, it is possible to conclude that~$S_t \leq c_7t\mu_3$ by considering two separate
cases depending on the value of~$\sigma^2$, as before. Hence, there are at
most~$c_7 t\mu_3 = O(bt^3 / n^4)$ copies of~$C_4$ in~$G$.

\begin{claim}
  \label{lm:triangle}
  With high probability,~$G$ contains at most~$cbt^2 / n^3$ triangles, for some
  constant $c$.
\end{claim}
\begin{claimproof}
  Given that~$t \ll n^{4 / 3}$, we may assume that~$b = \omega\paren{n^3 / t^2}$.
  Let~$X_1, \dots, X_{2t}$ be independent copies of a random variable that, given a
  uniformly at random pair~$e$ from~$\binom{[n]}{2}$, it counts the number of
  triangles in~$G + e$ that contain the edge~$e$. Observe that, we
  have~$\mu = \Expect{X_i} \leq c \cdot nb / t \cdot \paren{t / n}^2 \cdot 1 / n^2$
  since, this expectation is maximized when the number of copies of~$P_3$ is also
  maximum. This can be achieved by sorting the set of vertices by their degrees
  in~$G_t$ in non-increasing order, and then, as long as there is some budget left,
  buy all the edges incident to the current vertex~$v$ in consideration. Clearly
  every vertex~$v$ of~$G_t$ \whp satisfies~$t / 3n \leq d_{G_t}(v) \leq 3t /
  n$. Every purchased edge increases the degree of two vertices, and the graph~$G$
  contains~$b$ edges. Therefore, it is possible to buy all edges from a set of vertices of size at
  most~$2\cdot 3b / \paren{t / n} = 6bn / t$. Each vertex~$v$
  generates~$\binom{d_{G}(v)}{2}$ copies of~$K_3$ and~$P_3$. Now, summing over all
  vertices, the number of copies of~$K_3$ and~$P_3$ is at
  most~$3\cdot 6bn / t \cdot \paren{3t / n}^2 \leq 200 \cdot nb / t \cdot \paren{t /
    n}^2$. Given that, in the above experiment, we choose an edge uniformly at
  random, the expectation
  satisfies~$c_1 \cdot nb / t \cdot \paren{t / n}^2 \cdot 1 / n^2 \leq \mu \leq c
  \cdot nb / t \cdot \paren{t / n}^2 \cdot 1 / n^2 = c \cdot bt / n^3$.

  An algorithm that generates~$G$, at any point buys at most~$b$ edges. Given that
  the probability of generating a triangle is monotone with respect to edge set
  inclusion, the probability is maximized with~$G$, i.e., when the algorithm buys
  all~$b$ edges. The objective is to conclude that any algorithm \whp, after being
  offered to buy~$b$ out of~$t$ random edges, can generate at
  most~$c\cdot t\cdot \mu$ triangles with budget $b$, for some constant
  $c$. That is, the graph~$G$ contains at most~$O\paren{ct\mu}$ triangles.

  We get this by applying \Cref{lem:bennett_inequality} for the sequence of~$2t$
  independent random variables~$X_i$, where the pairs can be sampled multiple
  times. Between the sampled~$2t$ random pairs, \whp there are at least~$t$ different
  edges. Denote by~$S_t = \sum_{i=1}^t (X_{i} - \mu)$. Now we bound
  variance~$\sigma^2 = \sum_{i=1}^t \Expect\paren{X_{i} - \mu}^2$. 
  Since in the currently considered case we have $t=O(n^{4/3})$, where the exponent in $n$ is less than $3/2$, we can apply~\Cref{claim:every_edge_in_at_most_c_Ns} to get a constant $c_\eps$ with a bound that each pair of vertices $u, v$ is \whp in at most $c_\eps$ neighborhoods.
  Some edge $(u, v)$ creates a triangle if and only if it is placed into the neighborhood of some vertex in $G^t$.
  Therefore we can conclude that \whp $X_i \leq c_\eps$.
  Given that, for all~$i$,~$X_i$ attains only integer values, and
  we have~$X_i \leq c_\eps$, it follows
  that~$\mu / c_\eps \leq \Prob{X_i > 0} \leq \mu$ and~$|X_i - \mu| \leq c_\eps = a$. Therefore, we have
  that~$\sigma^2 \leq t \cdot c_\eps^2\mu$. We now apply \Cref{lem:bennett_inequality}
  to obtain
  \begin{align*}
    \Prob{S_t > ct\mu} \leq 2\exp\paren{-\frac{\sigma^2}{a^2}\cdot h\paren{\frac{act\mu}{\sigma^2}}},
  \end{align*}
  where $h(x) = (1+x)\log(1+x)-x$. From our bounds on~$\sigma^2$, we have that, for a
  sufficiently large constant~$c_3$,$ac_3t\mu / \sigma^2 \geq 3$,
  and~$h\paren{ac_3t\mu / \sigma^2} \geq ac_3t\mu / \sigma^2$. Finally, we have the bound
  \begin{align*}
    \Prob{S_t > c_3t\mu} \leq 2\exp\paren{-\frac{\sigma^2}{a^2}\cdot h\paren{\frac{ac_3t\mu}{\sigma^2}}} \leq \exp\paren{\frac{-c_3t\mu}{a}} = o(1),
  \end{align*}
  where the last equality holds by our bound on~$\mu$ and by our choice of~$b$.
  Let $q = \sum_{i=1}^{2t}X_i$. Then, we have
  that~$q - t\mu = S_t \leq c_3t\mu$.  Thus,~$q \leq (c_3+1)t\mu$. 
  Since $q$ is the number of generated triangles, then \whp~$G$ contains at
  most~$c_3t\mu \leq cc_3bt^2 = c'bt^2 /
  n^3$ triangles.
\end{claimproof}

Now we bound the number of copies of~$K_3^+$ in~$G$. Since~$b = \omega\paren{n^3 / t^2}$, by~\Cref{lm:triangle} the number of copies of~$K_3$ in~$G$ is at
most~$O\paren{bt^2 / n^3}$. Then, as every vertex of~$G$ has
degree~$O(t / n)$, every copy of a triangle is part of at most~$O(t / n)$ copies
of~$K_3^+$. Hence,~$G$ contains~$O\paren{bt^3 / n^4}$ copies of~$K_3^+$.
 
\noindent \textbf{Long time with short budget.} Let us now consider the case
where~$t = \omega\paren{n^{7 / 5}}$. In this case we have
that~$b \ll b^* = n^{4 / 3} / t^{2 / 3}$. Let~$G$ be a spanning subgraph of~$G_t$
with~$b$ edges. Similarly as in the previous case, we bound the number of copies
of~$C_4$ and~$K_3^+$.

Notice that we can count the number of copies of~$C_4$ by summing over each pair of
vertices and squaring their common neighborhood.
To upper-bound the number of copies of $C_4$ it is enough to sum the squares of common neighborhoods over all pairs of vertices.
We will show that \whp any
vertex~$v$ satisfies~$\sum_{u\in V(G), u\neq v}|N(v) \cap N(u)|^2 < c$ for some constant~$c$.
The probability that a random edge increases the size of a common neighborhood of some two vertices is increasing with adding more edges into the graph.
Since any algorithm can buy at most $b$ edges, we can upper-bound the probability by having exactly $b$ edges in the graph $G_B$.
Clearly, there are at most~$2b$ non-isolated vertices, say, those in~$B$.
Notice that when we increase the size of common neighborhood of some vertices $u$ and $v$, by adding one edge, then before adding that edge there must be a vertex $x$, which is connected to $u$ or $v$.
Specifically, the vertex $x$ has to have at least one edge already adjacent to it.
If some vertex $v$ by adding edge $e$ increases its degree above 1, then before adding edge $e$ it must hold $v\in B$.
Therefore we can just focus on the subgraph consisting of $b$ vertices in $B$, denoted as $G_B$.
We analyze random graph $G_B$, where we add each edge with original probability $p=\Theta(t/n^2)$.
Due to the conditions imposed on $b$ it holds $p = o(1/b)$.
If there is a random graph on $b$ vertices and each edge is sampled with probability $o(1/b)$, then \whp there exists a constant $c$ such that each vertex $v\in G_B$ has degree at most $c$.
If we fix some vertex $v$ then $\sum_{u\in V(G), u\neq v}|N(v) \cap N(u)|^2 \leq d_{G_B}(v)^2 \leq c^2$.
Therefore \whp in~$G$, every vertex~$v$ is contained in at most $c^2$ distinct copies of $C_4$.
Any algorithm can buy at most $b$ edges, therefore there are at most $2b$ vertices that are part of at least one copy of $C_4$.
Then, \whp~$G$ contains at most~$c^2\cdot 2b = o(n^2 / t)$ copies of~$C_4$.

To bound the number of copies of~$K_3^+$, notice that that, in order to maximize the
number of triangles, we can assume there is a vertex~$v$ with degree~$O(b)$. Then,
the expected number of triangles in~$G$ is~$O\paren{b^2 t / n^2}$.
Also, observe
that, since the sum of all degrees that belong to a triangle is~$O(b)$, there are at
most~$b$ copies of the graph~$K_3^+$ that can be obtained from a single triangle.
Then, \whp there are at
most~$O\paren{b^3 t / n^2} = o\paren{\paren{n^4 / t^2 \cdot t} / n^2} = o(n^2 / t)$
copies of~$K_3^+$ in~$G$.
In this case candidate edges come from the graph $K_3^+$, therefore the number of copies of $K_3^+$ is also an upper bound for the number of candidate edges.

Given that~$G$ contains at most~$o(n^2 / t)$ copies of~$C_4$ and~$K_3^+$, by a
similar argument as in the previous case, \whp~$G$ contains no~$K_4^-$.

\section{Proof of Theorem~\ref{thm:k-butterfly}}

\Cref{thm:k-butterfly} for~$k = 1$, was proved in \cite[Theorem
6]{frieze2023}. Hence, we assume that~$k \geq 2$.

\subsection{\texorpdfstring{$1$}{1}-statement}

Let~$t = \omega\paren{n^{\frac{4k - 1}{3k}}}$, and let~$b \gg b^*$,
where~$b^* = \max\set{n^{4k - 1} / t^{3k - 1}, n / \sqrt{t}}$.

\noindent \textbf{Short time with large budget.} We first consider the case
where~$n^{\frac{4k - 1}{3k}} \ll t = O\paren{n^{4 / 3}}$. In this case we have
that~$b \gg b^* = n^{4k - 1} / t^{3k - 1}$. Let~$T = t / (k + 1)$, and
let~$r = \omega\paren{n^{4k} / T^{3k}}$ such that~$rt / n = o(b)$.

Roughly speaking, the strategy in this case is to pick a set of~$r$ vertices, and
then reveal the first~$T$ edges of the graph, buying every edge that intersects these
selected vertices as long as the neighborhoods have size~$\Theta(t / n)$. We then
iteratively reveal~$T$ edges at a time, buying edges inside the neighborhoods and,
in each iteration deleting from the initial set of~$r$ vertices, those whose
neighborhoods were not hit by any revealed edge. After~$k$ of these iterations, we
still preserve~$\Omega\paren{r T^{3k} / n^{4k}} = \omega\paren{1}$ of these vertices, and each one
contains at least~$k$ disjoint edges in their neighborhoods, which generates the
graph~$T_k$. We now formalize this strategy.

Let~$R$ be an arbitrary set of~$r$ vertices of the graph~$G_0$. We first reveal the
first~$T$ edges and buy every edge that intersects~$R$, as long as every vertex
from~$R$ has at most~$T / 2kn$ previously bought edges. Let~$G^R$ be the graph
obtained as a result from this step. By construction,~$G^R$ contains at
most~$r \cdot \paren{T / 2kn + 1} = o(b)$ edges. By \Cref{lem:degree_concentration},
\whp every vertex~$v \in R$ satisfies~$d_{G_T}(v) \geq T / 3n \geq T / 2kn$, and
hence,~$d_{G^R}(v) \geq T / 2kn$ for all~$v \in R$. Thus, each set~$N_{G^R}(v)$ has
size~$T / 2kn$. We first make the following claim, whose proof we omit, as it is
analogous to the proof of \Cref{claim:every_edge_in_at_most_10_Ns}.
\begin{claim}
  \label{claim:at_most_10_neigh_T_k}
  With high probability, every pair~$e \in \binom{[n]}{2}$ is contained in at
  most~$10$ of the sets~$N_{G^R}(v_i)$.
\end{claim}

We now perform operations iteratively~$k$ times. We let~$R^{\paren{0}} = R$ be a set
of size~$r^{\paren{0}} = r$, and let~$G_{R^{\paren{0}}} = G_R$. In the~$i$-th round,
for~$i = 1, \dots, k$, we start with the graph~$G_{R^{\paren{i - 1}}}$ reveal~$T$
edges. For an edge~$e$ that was revealed in this process, we buy~$e$
if~$e \subset N_{G^R}(v)$ for some~$v \in R^{\paren{i - 1}}$, and~$e$ is disjoint from any
edge in~$E\paren{G^{R^{\paren{i - 1}}}[N_{G^R}(v)]}$, ie. the set of previously
bought edges inside this neighborhood. We do this for every revealed edge, so long as
not more than~$b / (k + 1)$ edges are bought in this round. After this, we
let~$R^{\paren{i}} \subset R^{\paren{i - 1}}$ be the set of vertices from~$R$ for which an
edge was bought inside its neighborhood, and let~$r^{\paren{i}} =
|R^{\paren{i}}|$. We also let~$G^{R^{\paren{i}}}$ be the graph obtained from this
round.

Fix~$i = 1, \dots, k$. Given a uniformly at random pair~$e \in \binom{[n]}{2}$, we
will define the random variables~$X_i$ and~$Z_i$ as follows. Let~$X_1, \cdots, X_T$ be
independent copies of a random variable that counts the number of
vertices~$v \in R^{\paren{i - 1}}$ such that~$e \subset N_{G^R}(v)$, and such
that~$e$ is disjoint from every edge of~$G^{R^{\paren{i - 1}}}[N_{G^R}(v)]$.
Let~$X = X_1 + \cdots + X_T$, i.e., we perform this experiment~$T$ times. Under this
setting, and for every vertex~$v \in R^{\paren{i - 1}}$, let~$Z_v$ be the number of
pairs from~$\binom{N_{G^R}(v)}{2}$ that were sampled.

For a given vertex~$v \in R^{\paren{i - 1}}$, observe that, given that there are at
most~$k$ edges in the graph~$G^{R^{\paren{i - 1}}}[N_{G^R}(v)]$, there
are~$\Theta\paren{2k \cdot |N_{G^R}(v)|} = \Theta\paren{T / n}$ pairs
in~$\binom{N_{G^R}(v)}{2}$ that are non-disjoint from some edge of the
graph~$G^{R^{\paren{i - 1}}}[N_{G^R}(v)]$. Hence, we
have~$\Expect{X_1} = \Theta\paren{\paren{r^{\paren{i -
        1}}\paren{\paren{\frac{T}{2kn}}^2 - \frac{T}{n}}} / n^2} =
\Theta\paren{r^{\paren{i - 1}} T^2 / n^4}$.  Then,
\begin{align*}
  \mu = \Expect{X} = T \Expect{X_1} = \Theta\paren{r^{\paren{i - 1}} T^3 /
  n^4}.
\end{align*}
We claim that \whp~$X \geq \mu / 2$. Indeed, by
\Cref{claim:at_most_10_neigh_T_k},~$X_i$ \whp is at most~$10$. We also have
that~$\Prob{X_i \geq 1} \leq \Expect{X_i} = o(1)$. Therefore, we have
that~$\sigma^2 = \sum_{i = 1}^T \Expect{X_i^2} \leq 100\mu (1 + o(1))$. Hence, by using
\Cref{lem:bennett_inequality} with~$\ell = \mu / 2$, we have
that~$\Prob{|X - \Expect{X}| \geq \mu / 2} = o(1)$.

Consider a set~$E \subset \binom{[n]}{2}$ of~$t$ pairs of vertices. For a uniformly
random pair~$e \in \binom{[n]}{2}$, let~$Y_1, \dots, Y_T$ be independent copies of a
random variable defined
as~$\mathbbm{1}_{\set{e \in E}} \cdot |\set{v \in R : e \subset N_{G^R}(v)}|$. Then,
by \Cref{lem:bennett_inequality}, it is possible to show that
\whp~$Y_1 + \cdots + Y_T = o(\mu)$. Also, note that,
since~$T = O\paren{n^{4 / 3}} = O(n^{7 / 5})$ we can use
\Cref{claim:sample_T_at_most_5_times} and \Cref{claim:at_most_10_neigh_T_k} to
conclude that,
\whp~$\sum_{v \in R^{\paren{i - 1}}} Z_v \geq \frac{1}{50} \paren{\sum_{i = 1}^T X_i}
- \sum_{i = 1}^T Y_i \geq \mu / 50 - o(\mu) \geq \mu / 100 = \Omega\paren{r^{\paren{i
      - 1}} T^3 / n^4}$. In the case that there is a vertex~$v \in R^{\paren{i - 1}}$
such that~$e\paren{G_{R^{\paren{i - 1}}}[N_{G^R}(v)]} \geq k$, clearly the
graph~$T_k$ is already constructed. If this is not the case, then we
have~$r^{\paren{i}} \geq \frac{1}{k} \sum_{v \in R^{\paren{i - 1}}} Z_v =
\Omega\paren{r^{\paren{i - 1}} T^3 / n^4}$. Finally, we have that,
\whp~$r^{\paren{k}} \geq r T^{3k} / n^{4k} \gg \paren{n^{4k} / T^{3k}} \paren{T^{3k}
  / n^{4k}} = 1$.

\noindent \textbf{Long time with short budget.} We now consider the case
where~$t = \omega\paren{n^{4 / 3}}$. In this case we have
that~$b \gg b^* = n / \sqrt{t}$. Let~$T = t / 2$.

Similar to the previous case, we now fix a vertex and, while revealing the first~$T$
edges, we buy each edge incident to it, obtaining a neighborhood of size~$O(T / n)$
and having purchased at most~$b / 2$ edges. We then reveal the remaining~$T$ edges,
and buy every edge contained in the neighborhood. Then, with high probability,~$k$
of those edges are disjoint, constructing a copy of~$T_k$. We now formalize this
strategy.

Let~$v$ be a vertex of the graph~$G_0$. We start by revealing the first~$T$ edges and
buy every edge incident to~$v$, as long as we buy at most~$b / 2$ edges. Let~$G^v$ be
the graph obtained in this step. Let~$N = N_{G^v}(v)$. By
\Cref{lem:degree_concentration}, \whp~$d_{G_T}(v) \geq T / 3n$. Therefore, we
distinguish the cases where~$|N| \geq T / 3n = \omega\paren{n^{1 / 3}}$,
and~$|N| = b / 2 = \omega\paren{n / \sqrt{T}}$, i.e., where we were and were not
restricted by the budget respectively. We proceed to reveal the remaining~$T$ edges
and buy each contained in~$N$, given that we only buy at most~$b / 2$ of them,
obtaining the graph~$B_t$. Note
that~$\Prob{kK_2 \subset B_t[N]} = \Prob{kK_2 \subset G_T[N]}$.

Note that~$T / n^2 \gg n^{-2 / 3}$. In the case
where~$|N| = \omega\paren{n^{1 / 3}}$, we have
that~$\paren{T / n^2} |N|^{2} \gg n^{-2 / 3} \paren{n^{1 / 3}}^{2} = 1$. In the case
where~$|N| = \omega\paren{n / \sqrt{T}}$, we have
that~$\paren{T / n^2} |N|^{2} \gg \frac{T}{n^2} \paren{\frac{n}{\sqrt{T}}}^{2} = 1$.
Thus, by \Cref{lem:threshold_graph_Gnm}, \whp~$kK_2 \subset G_T[N]$, and the strategy
succeeds. This completes the proof of the~$1$-statement.

\subsection{\texorpdfstring{$0$}{0}-statement}

If we assume that~$t = \omega\paren{n^{4 / 3}}$ and~$b = o\paren{b^*}$.
Then if~$b = o\paren{n / \sqrt{t}}$, by \cite[Theorem 6]{frieze2023}, \whp any strategy fails to build even a
triangle. Therefore, we may assume that~$n^{\frac{4k - 1}{3k}} \ll t \ll n^{4 / 3}$,
and~$b = o\paren{b^*}$. We then have that~$b = o\paren{n^{4k-1} / t^{3k - 1}}$. Let~$G$
be a spanning subgraph with~$b$ edges of the random graph~$G_t$.

\begin{claim}
  \label{lm:inductive}
  For a positive integer~$\ell < k$, let~$V_\ell$ be the set of vertices of~$G$ that are
  the central vertex of a copy of~$T_{\ell}$ in~$G$. Then,
  \whp~$|V_\ell| \leq c_\ell bt^2 / n^3 \cdot \paren{t^3 / n^4}^{\ell - 1}$, for some constant
  $c_\ell$.
\end{claim}
\begin{claimproof}
  We proceed by induction on $\ell$. For~$\ell = 1$, i.e., when~$T_\ell$ is a triangle, we
  know from \Cref{lm:triangle} that \whp~$G$ contains at most~$cbt^2 / n^3$
  triangles, which cover at most~$3cbt^2 / n^3$ distinct vertices.
  Therefore~$|V_1| \leq 3c \cdot \frac{bt^2}{n^3}$.

  Suppose that the statement holds for all smaller values than~$\ell > 1$, and suppose
  that~$|V_{\ell-1}| \leq c_{\ell-1}b \paren{t^2 / n^3}^{\ell - 1}$. Clearly every the central
  vertex of a copy of~$T_{\ell}$ is also the central vertex of a copy
  of~$T_{\ell - 1}$. We can assume
  that~$b = \omega\paren{n^{3 + 4(\ell-1)} / t^{2 + 3(l-1)}}$. With this, observe that for
  all~$\ell < k$,~$|V_\ell| = \omega(1)$. Indeed, because the size of the set~$V_\ell$ is non-increasing
  as~$\ell$ increases and~$|V_{k-1}| = \omega(1)$.

  Let~$X_1, \dots, X_{2t}$ be independent copies of a random variable that, given a
  uniformly at random pair~$e = \set{u, v}$ from~$\binom{[n]}{2}$, it counts the
  number of vertices $w$ from~$V_{\ell - 1}$ that became central vertices for some copy of~$T_\ell$ in~$G_t$ by adding the edge $e$. 
  If we choose some vertex $v\in V_{\ell - 1}$, there are only two ways to add some edge such that a new copy of $T_\ell$ is formed with vertex $v$ as a central vertex.
  Either we can add some edge to the neighborhood of vertex $v$ or we can add some edge that connects the endpoints of the path $vv_1v_2$, where $v_1,v_2 \in V(G)$.
  Given that \whp every vertex~$v$
  satisfies~$t / 3n \leq d_{G_t}(v) \leq 3t / n$, there are at most $\paren{3t / n}^2$ edges in the neighborhood of some vertex $v\in V_{\ell - 1}$ and there are at most $\paren{3t / n}^2$ paths of length 2 with starting point $v$.
  Thus for each vertex from~$V_{\ell-1}$ there are
  at most~$2\paren{3t / n}^2$ edges that could possibly complete a copy of $T_\ell$ with central vertex $v$.
  Therefore \whp
  \begin{align*}
    \mu = \Expect{X_i} \leq c |V_{\ell-1}|  \paren{t / n}^2 \paren{1 / n^2} \leq cb \paren{t^2 /
    n^3} \paren{t^3 / n^4}^{\ell - 2} \cdot \paren{t^2 / n^4}.
  \end{align*}
  We will prove that~$G$ contains at most~$ct\mu$ copies of~$T_{\ell - 1}$, for some
  constant~$c$. We do this by applying \Cref{lem:bennett_inequality} over the random
  variables~$X_1, \dots, X_{2t}$. Observe that when choosing these random pairs, \whp
  there are at least~$t$ distinct pairs.
  Let~$S_t = \sum_{i=1}^t \paren{X_{i} - \mu}$. Now we bound
  variance~$\sigma^2 = \sum_{i=1}^t \Expect\paren{X_{i} - \mu)}^2$. With high probability, for
  each pair~$uv$, there are at most 5 common neighbors of~$u$ and~$v$ in $G_t$,
  therefore each edge is adjacent to at most 5 vertices from~$V_{\ell-1}$.
  Thus it completes a copy of $T_\ell$ by inserting into the neighborhood for at most 5 vertices from $V_{\ell - 1}$.
  Also, each edge can complete a copy of $T_\ell$ by connecting the endpoints of path of length 2 for at most 2 vertices from $V_{\ell - 1}$.
  Hence \whp we can bound $X_i \leq 7$.
  Given that~$X_i$ takes only integer values, and~$X_i \leq 7$, we have
  that~$\mu / 7 \leq \Prob{X_i > 0} \leq \mu$ and~$|X_i - \mu| \leq 7 = a$. Thus, we have
  that~$\sigma^2 \leq 49t\mu$. Let $c$ be some constant. We now apply \Cref{lem:bennett_inequality} to obtain
  \begin{align*}
    \Prob{S_t > ct\mu} \leq \exp\paren{-\frac{\sigma^2}{a^2}\cdot h\paren{\frac{act\mu}{\sigma^2}}},
  \end{align*}
  where $h(x) = (1+x)log(1+x)-x$. We also have that, for a sufficiently large
  constant~$c$,~$\frac{act\mu}{\sigma^2} \geq 3$. Therefore~$h\paren{act\mu / \sigma^2} \geq act\mu / \sigma^2$. With this, we
  get
  \begin{align*}
    \Prob{S_t > ct\mu} \leq \exp\paren{-\frac{\sigma^2}{a^2}\cdot h\paren{\frac{act\mu}{\sigma^2}}} \leq \exp\paren{-\frac{ct\mu}{a}} = o(1).
  \end{align*}
  Last equality comes from the fact that $t\mu = \omega(1)$. Let $q = \sum_{i=1}^{2t}X_i$. Then, we have
  that~$q - t\mu = S_t \leq ct\mu$.  Thus,~$q \leq (c+1)t\mu$. Since $q$ is an upper bound for the number of distinct vertices that can be a central vertex for~$T_\ell$ then, \whp~$G$
  satisfies~$|V_\ell| \leq q \leq (c+1)t\mu \leq c_\ell b \paren{t^2 / n^3} \paren{t^3 / n^4}^{\ell-2} \paren{t^2
    / n^4} \cdot t = c_\ell \paren{bt^2 / n^3} \paren{t^3 / n^4}^{\ell - 1}$.
\end{claimproof}

To construct the graph~$T_k$, the algorithm must first construct the
graph~$T_{k - 1}$. By \Cref{lem:degree_concentration}, every vertex~$v$ of~$G_t$ \whp
has degree at most~$3t / n$. By using \Cref{lm:inductive} with~$\ell = k - 1$, the
set~$V_{k - 1}$
satisfies~$|V_{k - 1}| \leq c_\ell bt^2 / n^3 \cdot \paren{t^3 / n^4}^{k - 1}$. Combining this
with the degree condition above, we have that every vertex from~$V_{k - 1}$ creates
at most~$10t^2 / n^2$ candidate pairs of vertices, say those in~$E_f$, that create
the graph~$T_k$. Thus, summing up this quantity over all vertices in~$V_{k - 1}$, we
have
that~$|E_f| \leq cbt^2 / n^3 \paren{t^3 / n^4}^{k - 2} \paren{t^2 / n^2} = cb
\paren{t^{3k - 2} / n^{4k - 3}}$.

Clearly, if the graph~$T_k$ is constructed, then at some point the last edge of such
graph was purchased. For all~$t' > 0$, The probability that such edge arrives at
time~$t'$ is at most~$|E_f| / |E_{t'}|$, where~$E_{t'}$ is the set of all possible
edges that we can obtain in the random graph process at time~$t'$.
Then~$|E_{t'}| \geq \binom{n}{2} - t' \geq n^2 / 4$, for large enough~$n$. We then have
that
\begin{equation*}
  \frac{|E_f|}{|E_{t'}|} \leq cb \cdot \frac{t^{3k - 2}}{n^{4k - 3}} \cdot \frac{1}{n^2} = cb
  \cdot \frac{t^{3k - 2}}{n^{4k - 1}} = o(1 / t).
\end{equation*}
Therefore, the last edge of~$T_k$ arrives at time~$t$ with probability~$o(1)$. This
completes the proof of the~$0$-statement.

\section{Concluding remarks}
% In this paper, we studied the problem of constructing small graph with multiple cycles within the framework of the controlled random graph process introduced in \cite{frieze2023}. In particular, we characterize the
% optimal strategies for purchasing $K_4^-$ which is the first open case for the small graph proposed in \cite{frieze2023}. Moreover, we further investigated a class of multi-cyclic graphs, the $k$-butterfly, demonstrating that for some graph containing an arbitrary number of cycles, we can determine an optimal purchasing strategy based on time.
Notably, coming up with the optimal strategy for constructing the graph~$K_4^-$ is
more challenging than for the graph~$T_k$. In fact, the density of the graph
(specifically, the maximum local edge density) is a critical factor in determining
the difficulty of such problems, particularly for the~$1$-statement, which
establishes the upper bound. At the same time, we believe that our proof of
the~$0$-statement could potentially be further generalized and extended to arbitrary
small graphs. Consequently, a natural question is to come up with an optimal strategy
for generating cliques of arbitrary size.
\begin{problem}
  Let $m \ge 4$, and let~$t = \omega\paren{n^{\frac{2m-4}{m-1}}}$, so that~$G_t$ \whp
  contains a copy of $K_m$. For which values of~$b$ does Builder have
  a~$(t, b)$-strategy that successfully constructs a copy of~$K_m$ with high
  probability? Conversely, for which values of~$b$ does Builder fail with high
  probability to construct a copy of~$K_m$?
\end{problem}

\end{document}